\documentclass[12pt,a4paper]{article}
\usepackage[utf8]{inputenc}
\usepackage{amsmath}
\usepackage{amsfonts}
\usepackage{amssymb}
\usepackage{graphicx}
\usepackage{mathabx}

    \usepackage{pdfsync}

\newtheorem{lemma}{Lemma}[section]

\newtheorem{proposition}{Proposition}[section]

\newenvironment{proof}[1][Proof]{\textbf{#1.} }{\ \rule{0.5em}{0.5em}}
\author{Yves Le Jan}
\title{Markov loops, coverings and fields } 
\begin{document}
\maketitle

\begin{abstract}
We investigate the relations between the Poissonnian loop ensembles, their occupation fields, non ramified Galois coverings of a graph, the associated gauge fields, and random Eulerian networks.
\end{abstract}

\section{Introduction}
Relations between occupation fields of Markov processes and Gaussian processes have been the object of many investigations since the seminal work of Symanzik \cite{Symanz} in which Poisson ensembles of Brownian loops were implicitly used. Since the work of Lawler and Werner \cite{LW} on "loop soups", these ensembles have also been the object of many investigations. Their properties can be studied in the context of rather general Markov processes. The purpose of the present work is to explore new directions in this context, in particular the relation with gauge fields.
\section{Discrete Topology}
\subsection{Graphs and fundamental groups }
 In this first section, we will briefly present the topological background of our study.


Our basic object will be a graph $\mathcal{G}$, i.e. a set of vertices $X$ together with a set of non oriented edges $E$. \textsl{We assume it is connected, and that there is no loop-edges nor multiple edges (though this is not really necessary)}. The set of oriented edges
is denoted $E^{o}$. It will always be viewed as a subset of $X^{2}$, without
reference to any imbedding.
 An oriented edge $(x,y)$\ is defined by the choice of an
ordering in an edge$.$ We set $-(x,y)=(y,x)$ and if $e=(x,y)$, we denote it
also $(e^{-},e^{+})$. The degree $d_{x}$ of a vertex $x$ is by definition the
number of non oriented edges incident at $x$.

\index{geodesic}
A $n$-tuple of elements of $X$, say $(x_{0},x_{1},...,x_{n})$ is called a path on $X$ 
  iff $\{x_{i},x_{i+1}\}\in E$ (path segment on
the graph) for all $i$ and  a \emph{geodesic arc} if moreover $x_{i-1}\neq x_{i+1}$ (no backtracking). Geodesic arcs
starting at $x_{0}$ form a \emph{marked tree} $\mathfrak{T}_{x_{0}}$ rooted in
$x_{0}$ (if we identify $x_{0}$ with the path $(x_{0})$. The marks belong to $X$: they are the endpoints of the geodesic
arcs, thus we have a canonical projection $p$ from $\mathfrak{T}_{x_{0}}$ onto $X$.  Oriented edges of $\mathfrak{T}_{x_{0}}$ are defined by pairs of
geodesic arcs of the form:
\index{universal covering}
$((x_{0},x_{1},...,x_{n}),(x_{0},x_{1},...,x_{n},x_{n+1}))$ (the orientation
is defined in reference to the root). $\mathfrak{T}_{x_{0}}$ is a
\emph{universal covering} of $X$ \cite{Mass}.

A (discrete) loop based at $x_{0}\in X$ is by definition a path $\xi=(\xi
_{1},...,\xi_{p(\xi)})$, with $\xi_{1}=x_{0}$, and $\{\xi_{i},\xi_{i+1}\}\in
E$, for all $1\leq i\leq p$ with the convention that $\xi_{p+1}=\xi_{1}$. On the
space of geodesic loops based at some point $x_{0}$,
we can define an operation (by  concatenation and cancellation of two inverse subarcs) which yields a group structure (the neutral element is the empty loop) $\Gamma_{x_{0}}$.  Note that the
fiber of the universal covering $\mathfrak{T}_{x_{0}}$ at $x_{0}$ is $\Gamma_{x_{0}}$.

There is a natural left action of $\Gamma_{x_{0}}$\ on $\mathfrak{T}_{x_{0}}$.
It can be interpreted as a change of root in the tree (the new root having the same mark $x_0$). Note that $X=\Gamma_{x_{0}%
}\backslash\mathfrak{T}_{x_{0}}$ (here we use of the quotient on the left
corresponding to the left action).
Besides, any geodesic arc between $x_{0}$ and another point $y_{0}$ of $X$
defines an isomorphism between $\mathfrak{T}_{x_{0}}$ and $\mathfrak{T}%
_{y_{0}}$ (change of root, with different root marks) .


 The groups $\Gamma_{x_{0}},x_{0}\in
X$\ are conjugated in a non canonical way.
The structure of
$\Gamma_{x_{0}}$ does not depend on the base point and this isomorphism class defines the
\index{fundamental group}\emph{fundamental group} $\Gamma$ of the graph (as the graph is connected: see for example \cite{Mass}).

\bigskip

A \emph{spanning tree} $T$\ is by definition a subgraph of $\mathcal{G}$ which is a
tree and covers all points in $X$.
\index{spanning tree} 
It has necessarily $\left|  X\right|
-1$\ edges.\\
 The inverse images of a spanning tree by the canonical projection from a
universal cover $\mathfrak{T}_{x_{0}}$\ onto $X$ form a tesselation on
$\mathfrak{T}_{x_{0}}$, i.e. a partition of $\mathfrak{T}_{x_{0}}$\ in
identical subtrees, which are fundamental domains for the action of
$\Gamma_{x_{0}}$. Conversely, a section of the canonical projection from the
universal cover with connected image defines a spanning tree.

Fixing a spanning tree determines a unique geodesic between two points of $X$.
Therefore, it determines\ the conjugation isomorphisms between the various
groups $\Gamma_{x_{0}}$.
%

\emph{The fundamental group }$\Gamma$\emph{ is a free group} with $\left|
E\right|  -\left|  X\right|  +1=r$ generators. To construct a set of
generators, one considers a spanning tree $T$\ of the graph, and choose an
orientation on each of the $r$ remaining links. This defines $r$ oriented
cycles on the graph and a system of $r$ generators for the fundamental group.
(See \cite{Mass} or Serre (\cite{Ser}) in a more general context).

\index{reduced path}Given any finite path $\omega$ with starting point $x_{0}$, the reduced path
$\omega^{R}$ is defined as the geodesic arc defined by the endpoint of the
lift of $\omega$ to $\mathfrak{T}_{x_{0}}$.

Tree-contour-like based loops can be defined as discrete based loops whose
lift to the universal covering are still based loops. Each link is followed
the same number of times in opposite directions (backtracking). The reduced
path $\omega^{R}$\ can equivalently be obtained by removing all
tree-contour-like based loops imbedded into it. In particular each loop $l$
based at $x_{0}$ defines an element $l^{R}$ in $\Gamma_{x_{0}}$.%

\subsection{Geodesic loops and conjugacy classes}
 Loops are defined as equivalence classes of based loops under the
natural shift $\theta$ defined by $\theta\xi=(\xi
_{2},...,\xi_{p(\xi)},\xi_{p(\xi)+1}=\xi_1)$, with  $\xi=(\xi
_{1},...,\xi_{p(\xi)})$.

Geodesic loops are of particular interest as they are in bijection with the set of conjugacy classes
of the fundamental group. Indeed, if we fix a reference point $x_{0}$, a geodesic
loop defines the conjugation class formed of the elements of $\Gamma_{x_{0}}$
obtained by choosing a base point on the loop and a geodesic segment linking
it to $x_{0}$. Any non trivial element of $\Gamma_{x_{0}}$ can be obtained in
this way.


Given a loop, there is a canonical geodesic loop associated with it. It is
obtained  by removing recursively all tail edges (i.e. pairs of consecutive inverse oriented edges of the
loop) .

\subsection{ Galois Coverings and Monodromy}
There are various non-ramified coverings, intermediate between $\mathcal{G}=(X,E)$ and the
universal covering. Non ramified means that locally, the covering space is
identical to the graph (same incident edges). 
More precisely, a graph $\widetilde{\mathcal{G}}=(\widetilde{X},\widetilde{E})$ is an non-ramified covering of $\mathcal{G}$ if there exist a map $p$ from $\widetilde{X}$ onto $X$ such  that, for every
vertex $u$ in $\widetilde{X}$, the projection $p$ restricts to a bijection from the set of neighbors of $u$ to
the set of neighbors of $p(u)$. We will consider only non-ramified coverings.

Then each oriented path segment
on $X$ can be lifted to the covering in a unique way, given a lift of its
starting point.

Each covering is (up to an isomorphism) associated with a
subgroup $\widetilde{\Gamma}$\ of the fundamental group $\Gamma$, defined up to conjugation. More precisely, given a covering $\widetilde{\mathcal{G}}$, a point $x_{0}$ of $X$ and a point
$\widetilde{x}_{0}$ in the fiber above $x_{0}$, the closed geodesics based at
$x_{0}$ whose lift to the covering\ starting at $\widetilde{x}_{0}$\ are
still closed form a subgroup $\widetilde{\Gamma}_{\widetilde{x}_{0}}$\ of $\Gamma_{x_{0}}$,
canonicaly isomorphic to the fundamental group of $\widetilde{\mathcal{G}}$ represented
by closed geodesics based at $\widetilde{x_{0}}$.

Conversely, if $\widetilde{\Gamma}_{x_{0}}$ is a subgroup of $\Gamma_{x_{0}}$, the covering is defined
as the quotient graph $(Y,F)$ with $Y=\widetilde{\Gamma}_{x_{0}}\backslash\mathfrak{T}_{x_{0}}$ and $F$
the set of edges defined by the canonical projection from $\mathfrak{T}%
_{x_{0}}$ onto $Y$.

If $\widetilde{\Gamma}_{x_{0}}$ is a normal subgroup, the quotient group (called the covering or the monodromy
group) $M_{x_{0}}=\widetilde{\Gamma}_{x_{0}}\backslash\Gamma_{x_{0}}$ acts faithfully on\ the fiber at $x_{0}$. We say the covering is a Galois (or normal) covering.\\
An example is the commutator subgroup $[\Gamma_{x_{0}},\Gamma_{x_{0}}]$. The
associate covering is the maximal Abelian covering at $x_{0}$.  The monodromy group is the first homology group $H_1(\mathcal{G},\mathbb{Z})$ of the graph. It is an Abelian group with $n=\vert E \vert -\vert X \vert +1$ generators.\\
Another example is the cube, which, by central symmetry, is a twofold covering of the tetrahedron
associated with the group $\mathbb{Z}/2\mathbb{Z}$.\\
Monodromy groups associated with different base points are then isomorphic. Any of them will be denoted $M$.\\
 Every based loop in $\mathcal{G}$ defines an element of $\Gamma_{x_{0}}$ and an element of the monodromy group $M_{x_{0}}$ at the base point whose conjugacy class is independent of the geodesic linking $X_0$ to the base point and invariant under a change of base point. It is unchanged if we erase all tail edges so that any conjugacy class of $\Gamma_{x_{0}}$, i.e. any geodesic loop $C$ determines a conjugacy class of $M$.

Each spanning tree of $\mathcal{G}$ determines a tesselation of $\widetilde{X}$, isomorphisms between the fibers of the covering, between different groups $\Gamma_{x_{0}}$ as $x_0$ varies in $X$, which induce isomorphisms between the groups $\widetilde{\Gamma}_{x_{0}}$ and the quotient groups $M_{x_{0}}$ (which are represented by the fibers). It follows that there is an action of $M$ on $ \widetilde{X}$ which preserves the tesselation such that $X = M\setminus \widetilde{X} $.\\

Given a finite group $M$, we can create a Galois covering of $\mathcal{G}$ by assigning to each oriented edge $(x,y)$ an element of $M$,  $U_{(x,y)}$ in such a way that opposite edges correspond to inverse elements.  It is associated with the subgroup of $\Gamma_{x_0}$ formed by geodesic loops based at $x_0$ such that the ordered product of the $U_{(x,y)}$ assigned to the edges of the loop is equal to the identity.

If we fix a base point, the monodromies of the loops based at form a subgroup $M'_{x_0}$ of $M$, and these subgroups are isomorphic by conjugacy if we change the base point. We can therefore reduce our attention to the case $M'=M$. The vertex set of the covering is then represented by $X\times M$. 

  Note that if we attach an element $m_x$ of $M$ to each vertex and replace $U_{(x,y)}$ by $m_x U_{(x,y)} m_y ^{-1}$, the covering is unchanged. In particular, if we choose a spanning tree of $\mathcal{G}$, the covering can determined by assigning to the edges of the spanning tree the identity and to the other edges the monodromies of the loops they determine. Fixing such a $M$-assignment can be also expressed as fixing a gauge field. The gauge group $M^X$ acts faithfully  by conjugacy on these assignments and $M$-coverings are the orbits of this action.
  
  \smallskip 
   
   Note that $M$-assignments are the counterpart, in discrete geometry, of the $\mathfrak{g}$-valued differential forms defining a connection on a $G$-principal bundle.

 Given a $M$- assignment $U$, we define the conjugacy class in $M$ of a loop $l$, denoted  $C_U(l)$ as the image of conjugacy class of any its representatives in $\Gamma$ by the canonical projection.  This conjugacy class depends only on the covering defined by $U$ and on the geodesic loop defined by $l$. It is the conjugacy class of the product of the elements of $M$ attached to the edges of the loop. 

\section{Markov loops}
\subsection{The loop ensemble and the free field}

We adopt the framework described in \cite{stfl}. Given a graph $\mathcal{G}=(X,E)$, a set of non negative conductances $C_{x,y}=C_{y,x}$ indexed by the set of edges $E$ and a non negative killing measure $\kappa$ on the set of vertices $X$, we can associate to them an energy (or Dirichlet form) $\mathcal{E}$, we will assume to be positive definite, which is a transience assumption. For any function $f$ on $X$, we have: $$\mathcal{E}(f,f)=\frac{1}{2}\sum_{x,y}C_{x,y}(f(x)-f(y))^{2}+\sum_x \kappa_x f(x)^{2}.$$
There is a duality measure $\lambda$ defined by $\lambda_x=\sum_y C_{x,y} +\kappa_x$.
Let $G_{x,y}$ be the symmetric Green's function associated with $\mathcal{E}$.\\
The  associated symmetric continuous time Markov process can be obtained from the Markov chain defined by the transition matrix $P_{x,y}=\frac{C_{x,y}}{\lambda_y}$ by adding independent exponential holding times of mean $1$ before each jump. If $P$ is submarkovian, the chain is absorbed at a cemetery point $\Delta$. If $X$ is finite, the transition matrix is necessarily submarkovian.\\
The complex (respectively real) free field is the complex (real) Gaussian field on $X$ whose covariance function is $G$. We will denote it by $\varphi$ (respectively $\varphi ^{\mathbb{R}}$).\\
We denote by $\mu$ the loop measure associated with this symmetric Markov process. It can also be viewed as a shift invariant measure on based loops. We can refer to \cite{stfl} for the general definition in terms of Markovian bridges, but let us mention that:\\
- the measure of a non-trivial discrete loop is the product of the transition probabilities of its edges if it is aperiodic; otherwise this product should be divided by the multiplicity of the loop.\\
- the measure on continuous time loops is then obtained by including exponential holding times, except for one point loops on which the holding time measure (which has infinite mass)  has density $\frac{e^{-t}}{t}$.

 The Poissonian loop ensemble $\mathcal{L_{\alpha}}$ is the Poisson process of loops of intensity $\alpha\mu$. It can be constructed in such a way that the the set of loops  $\mathcal{L}_\alpha$ increases with $\alpha$.\\ We set $\mathcal{L}=\mathcal{L}_1$ Recall that when $\mathcal{G}$ is finite, $\mathcal{L}$ can be sampled by Wilson algorithm (Cf: \cite{stfl}, \cite{chang}).\\

\subsection{Occupation fields }
 We denote by $\hat{\mathcal{L}}_{\alpha}$ the occupation field associated with $\mathcal{L}_\alpha$ i.e. the total time spent in $x$ by the loops of $\mathcal{L}_\alpha$, normalized by $\lambda_x$.  
It has been shown in $\cite{aop} $ (see also $\cite{stfl}$) that the fields $\hat{\mathcal{L}}=\hat{\mathcal{L}_1}$ ($\hat{\mathcal{L}}_{\frac{1}{2}}$) and $\frac{1}{2} \varphi^2$ ($\frac{1}{2} (\varphi^{\mathbb{R}})^2$) have the same distribution. Note that this property extends naturally to symmetric Markov processes in which points are non-polar and in particular to one dimensional diffusions (see \cite{Lupdif}). Generalisations to dimensions 2 and 3 involve renormalization (Cf \cite{stfl}).

   Note that a natural coupling of the free field with the occupation field of the loop ensemble of intensity $\frac{1}{2}\mu$ has been recently given by T. Lupu \cite{Lup}, using loop clusters.\\

   In what follows, we will assume for simplicity that $\mathcal{G}$ is finite. We will now define the edge occupation fields associated with the loop ensembles.\\
Given any oriented edge $(x,y)$ of the graph, denote by $N_{x,y}(l)$ the total number of jumps made from $x$ to $y$ by the loop $l$ and by $N^{(\alpha)}_{x,y}$ the total number of jumps made from $x$ to $y$ by the loops of $\mathcal{L}_{\alpha}$. Note that $N^{(\alpha)}_{x,x}=0$.
\\ Let $Z$ be any Hermitian matrix indexed by pairs of vertices and $\chi$ a non-negative measure on $X$.\\
The content of the following lemma appeared already in chapter 5 and 6 of \cite{stfl} (see remarks 11 and 13 for ii) and iii)).
\begin{lemma}\label{toto} 
Denote by $P^Z_{x,y}$ the matrix $P_{x,y}Z_{x,y}$.\\
\begin{enumerate}
\item[i)]  We have: $$E(\prod_{x\neq y} Z_{x,y}^{N^{(\alpha)}_{x,y}}e^{-\sum_x \chi_x \hat{\mathcal{L}}_{\alpha}^x} )=\left[\frac{\det(I-\frac{\lambda}{\lambda+\chi}P^{Z})}{\det(I-P)}\right]^{-\alpha}.$$
\item[ii)] For $\alpha=1$,  $$E(\prod_{x\neq y} Z_{x,y}^{N^{(1)}_{x,y}}e^{-\sum_x \chi_x \hat{\mathcal{L}}_{1}^x} ) =E(e^{\sum_{x\neq y}(\frac{1}{2} C_{x,y} (Z_{x,y}-1)\varphi_x \bar{\varphi}_y)}e^{-\frac{1}{2}\sum_x \chi_x\varphi_x\bar{\varphi}_x} ).$$
\item[iii)] For $\alpha=\frac{1}{2}$,  $$E(\prod_{x\neq y} Z_{x,y}^{N^{(\frac{1}{2})}_{x,y}}e^{-\sum_x \chi_x \hat{\mathcal{L}}_{\frac{1}{2}}^x} ) =E(e^{\sum_{x\neq y}\frac{1}{2} C_{x,y} (Z_{x,y}-1)\varphi ^{\mathbb{R}}_x\varphi ^{\mathbb{R}}_y}e^{-\frac{1}{2} \sum_{x}\chi_x (\varphi ^{\mathbb{R}}_x)^2} ).$$
\end{enumerate}
\end{lemma} 

\subsection{Eulerian networks}
We define a network to be a  $\mathbb{N}$-valued function defined on oriented edges of the graph. It is given by a matrix $k$ with  $\mathbb{N}$-valued coefficients which vanishes on the diagonal and on entries $(x,y)$ such that $\{x,y\}$ is not an edge of the graph. We say that $k$ is Eulerian if $$ \sum_y k_{x,y}= \sum_y k_{y,x}.$$ For any Eulerian network $k$, we define $k_x$ to be $\sum_y k_{x,y}=\sum_y k_{y,x}$.
It is obvious that the field $N^{(\alpha)}$ defines a random network which verifies the Eulerian property.\\

The distribution of the random network defined by $\mathcal{L}_{\alpha}$  was given in \cite{lejanito}. The cases $\alpha=1$ is of special interest:
\begin{proposition}
i) For any Eulerian network $k$,$$  P(N^{(1)}=k)=\det(I-P)\frac{\prod_x {k_x}!}{\prod_{x,y} k_{x,y}!} \prod_{x,y} P_{x,y}^{k_{x,y}}.$$
ii) For any Eulerian network $k$, and any nonnegative function $\rho$ on $X$  $$P(N^{(1)}=k\:,\: \hat{\mathcal{L}}_{1}\in (\rho, \rho+d\rho )=\frac{1}{\det(G)}\prod_{x,y} \frac{(\sqrt{\rho_x} C_{x,y}\sqrt{\rho_y})^{k_{x,y}}}{ k_{x,y}!}\prod_{x}\frac{1}{2}e^{-\frac{1}{2}\lambda_x\rho_x }d\rho_x.$$
\end{proposition} 

\begin{proof} 
i) was proved in two different ways in \cite{lejanito}.
For ii), the first proof of i) can be extended as follows:  Let $\mathfrak{N}$ be the additive semigroup of networks and $\mathfrak{E}$ be the additive semigroup of Eulerian networks.
 From the previous lemma, we get\\

$E(\prod_{x\neq y} Z_{x,y}^{N^{(1)}_{x,y}}e^{-\sum_x \chi_x \hat{\mathcal{L}}_{1}^x} ) =E(e^{\sum_{x\neq y}(\frac{1}{2} C_{x,y} (Z_{x,y}-1)\varphi_x \bar{\varphi}_y)}e^{-\frac{1}{2}\sum_x \chi_x \varphi_x\bar{\varphi}_x })
\\
= \frac{1}{(2\pi)^d \det(G)} \int e^{-\frac{1}{2}(\sum_{x}(\lambda_x+\chi_x) \varphi_x\bar{\varphi}_x -\sum_{(x,y)\in K\times K} C_{x,y} Z_{x,y}\varphi_x \bar{\varphi}_y)} \prod_x  \frac{1}{2i}d \varphi_x\wedge d\bar{\varphi}_x \\
\\
=\frac{1}{(2\pi)^d \det(G)} \int_0^\infty \int_0^{2\pi} e^{-\frac{1}{2}(\sum_{x}(\lambda_x +\chi_x)r_x^2 -\sum_{x,y} C_{x,y} Z_{x,y}r_x r_y e^{i(\theta_x-\theta_y)})}\prod_x r_x d r_xd{\theta}_x\\
\\
=\frac{1}{(\det(G)} \int_0^\infty \int_0^{2\pi} e^{-\frac{1}{2}\sum_{x}(\lambda_x+\chi_x) r_x^2 }\sum_{n\in \mathfrak{N}}\prod_{x,y\in K} \frac{1}{n_{x,y}!}(C_{x,y}(\frac{1}{2}Z_{x,y}r_x r_y e^{i(\theta_x-\theta_y)})^{n_{x,y}} \prod_x \frac{r_x}{2\pi} d r_xd{\theta}_x$.\\

Integrating in the $\theta_x$ variables and using the definition of Eulerian networks, it equals
\\
$\frac{1}{\det(G)} \int_0^\infty  e^{-\frac{1}{2}\sum_{x}(\lambda_x +\chi_x)r_x^2 }\sum_{n\in \mathfrak{E}}\prod_{(x,y)\in K\times K} \frac{1}{n_{x,y}!}(\frac{1}{2}C_{x,y} Z_{x,y}r_x r_y )^{n_{x,y}} \prod_x r_x d r_x$.\\ 
It follows that for any functional $F$ of a field on $X$,
$E(\prod_{x\neq y} Z_{x,y}^{N^{(1)}_{x,y}}F(\hat{\mathcal{L}}_{1} )) \\
 =\frac{1}{\det(G)} \int_0^\infty  e^{-\frac{1}{2}\sum_{x}(\lambda_x)r_x^2 }\sum_{n\in \mathfrak{E}}\prod_{(x,y)\in K\times K} \frac{1}{n_{x,y}!}(\frac{1}{2}C_{x,y} Z_{x,y}r_x r_y )^{n_{x,y}} F(r^2) \prod_x r_x d r_x$.\\ 
We conclude the proof of the proposition by letting $F$ be an infinitesimal indicator function and by identifying the coefficients of $\prod_{x,y} Z_{x,y}^{k_{x,y}}$.
\end{proof}

Note that given $N^{(1)}=k$, all $\frac{\prod_x {k_x}!}{\prod_{x,y} k_{x,y}!} $ discrete loops configurations are equally likely. 
Note also that from this proposition follows the Markov property extending the reflection positivity property proved in chapter 9 of \cite{stfl}: If $X$ is the disjoint union of  $X_1$ and $X_2$ and we condition $N_{x,y}$ and $N_{y,x}$ to take certain values for $x\in X_1$ and  $y\in X_2$, the restrictions of $N$ to $X_1 \times X_1$ and $X_2 \times X_2$ are independent.

For $\alpha=\frac{1}{2}$, denote $N^{(\frac{1}{2})}_{\{  \} } $ the field $N^{(\frac{1}{2})}_{\{ x,y \}} =N^{(\frac{1}{2})}_{x,y}+N^{(\frac{1}{2})}_{y,x}.$ Note that $\sum_y  N^{(\frac{1}{2})}_{\{ x,y \}}$ is always even.
We call even networks the sets of numbers attached to non oriented edges such that $k_x=\dfrac{1}{2}\sum_y k_{\{ x,y \}}$ is an integer.
Similarly, we have the following   
\begin{proposition} 
 i) For any even network $k$,$$  P(N^{(\frac{1}{2})}_{\{  \} } =k)=\sqrt{\det(I-P)}\frac{\prod_x {2 k_x}!}{\prod_x  2^{k_x} {k_x}! \prod_{x,y} k_{\{ x,y \}}!} \prod_{x,y} P_{x,y}^{k_{x,y}}.$$
 ii)  For any even network $k$, and any nonnegative function $\rho$ on $X$  $$P(N^{(\frac{1}{2})}_{\{  \} } =k\:,\: \hat{\mathcal{L}}_{\frac{1}{2}}\in (\rho, \rho+d\rho )=\frac{1}{\sqrt{\det(G)}} \prod_{x,y} \frac{(\sqrt{\rho_x} C_{x,y}\sqrt{\rho_y})^{k_{x,y}}}{ k_{x,y}!}\prod_{x}\frac{1}{\sqrt{2\pi \rho}}e^{-\frac{1}{2}\lambda_x\rho_x }d\rho_x.$$

\end{proposition} 
\begin{proof} 
Let  $\mathfrak{F}$ be the additive semigroup of even networks. 
 To prove i) note that on one hand, for any symmetric matrix $S$
$$ E(\prod_{\{ x,y \}} S_{x,y}^{N^{(\frac{1}{2})}_{\{ x,y \}}})=\sum_{k\in \mathfrak{F} } P(N^{(\frac{1}{2})}_{\{  \} } =k)\prod_{\{ x,y \}} S_{x,y}^{k_{\{ x,y \}}}.$$ 
On the other hand, from the previous lemma: \\
$ E(\prod_{\{ x,y \}} S_{x,y}^{N^{(\frac{1}{2})}_{\{ x,y \}}})=E(e^{\sum_{x,y}(\frac{1}{2} C_{x,y} (S_{x,y}-1)\varphi^{\mathbb{R}}_x \varphi^{\mathbb{R}}_y)} )\\
\\
= \frac{1}{(2\pi)^{d/2} \sqrt{\det(G)}} \int e^{-\frac{1}{2}(\sum_{x}\lambda_x {(\varphi^{\mathbb{R}}_x)}^2 -\sum_{(x,y)\in K\times K} C_{x,y} S_{x,y}\varphi^{\mathbb{R}}_x\varphi^{\mathbb{R}}_y)}\prod_x  d \varphi^{\mathbb{R}}_x\\$

and we conclude as before by expanding the exponential of the double sum and the expression of the moments of the normal distribution.
Then ii) follows in the same way as in the proof of the previous proposition.
\end{proof} \\
We can deduce from ii) that the symetrized $N^{(\frac{1}{2})}$ field conditionned by the vertex occupation field is, as it was observed by Werner
in \cite{wernersemiprob}, a random current model. 

A Markov property also holds (see \cite{wernersemiprob} and also \cite{camialis}  in the context of non backtracking loops).

\section{Fields and coverings}
\subsection{Decompositions}
 Given a covering $\widetilde{\mathcal{G}}=(\widetilde{X},\widetilde{E})$, the killing measure and the conductances are naturally defined on it so that they are invariant under the action of the monodromy group and they project on $C$ and $\kappa$.  
The Dirichlet form and the associated Markov process can be naturally lifted to any non ramified covering. 
 
 The Green functions of the covering denoted $\widetilde{G}$ is related to $G$ by the following identity:
$$G(p(u),p(v))=\sum_{m\in M}\widetilde{G}(u,m\cdot v).$$\\Let $\mathbb{I}$ be the identity element in $M$. If we fix a section of $p$,  the previous identity can be rewritten as follows:$$G(x,y)=\sum_{m\in M}\widetilde{G}((x,\mathbb{I}),(y,m).$$
From that, we deduce that if f $\vert M \vert$ is finite, the free field  of the covering denoted $\widetilde{\varphi}$ is related to $\varphi$ by the following identity:\\
$$\varphi\circ p(u) \,{\buildrel d \over =}\, \dfrac{1}{\sqrt{\vert M \vert}}\:\sum_{m\in M}\widetilde{\varphi}(m\cdot u).$$\\
 Define $\mathcal{L}^0_{\alpha}=\{l\in \mathcal{L_{\alpha}}, C_{U}(l)=\mathbb{I}\}$ Let $\{ \mathcal{L}^{0,m}_{\alpha},\: m \in M \} $ be independent copies of $\mathcal{L}^0_{\alpha}$.\\
Choose a fundamental domain $F$ in $\widetilde{X}$ to lift $\mathcal{L}^{0,\mathbb{I}}_{\alpha}$ (the base points being lifted to $F$) and lift each $\mathcal{L}^{0,m}_{\alpha}$ to $m(F)$. The union of these lifts is identical to $\widetilde{\mathcal{L}_{\alpha}}$ in distribution.\\

Given a $M$-assignment $U$, and an irreducible unitary representation $\pi$ of $M$, we define a tranfer matrix $P^{U,\pi}$on $X\times\mathbb{C}^{dim(\pi)}$: $$P^{U,\pi}_{(x,i),(y,j)}=P_{x,y}\pi(U_{x,y})_{ij}.$$ Let $G^{U,\pi}$ denote the  associated Green function, and $\varphi^{U,\pi}$ the associated vector free field.

From the decomposition of the regular representation into irreducible representations, we get the following decomposition of  $\widetilde{G}$.

$$\widetilde{G}((x,m),(y,n))=\sum_{\pi}\sum_{i,j=1}^{dim(\pi)} G^{U,\pi}_{(x,i),(y,j)}\pi(n^{-1}m)_{ij}.$$

 Let $\varphi^{U,\pi,j}$ be $dim(\pi)$ independent copies of $\varphi^{U,\pi}$. Define them jointly in $\pi$ so that they are independent Then we can deduce from the former decomposition of $\widetilde{G}$ that in distribution:  
 $$\widetilde{\varphi}((\cdot,m)) \,{\buildrel d \over =}\, \sum_{\pi}\sum_{i,j=1}^{dim(\pi)}\pi(m)_{ij}\varphi^{U,\pi,j}_i(\cdot).$$
 
 

 If we perform a change of gauge,  we see that the fields $\varphi^{U,\pi,j}$ are transformed consistently, therefore we see them as representatives in a particular gauge of  intrinsic fields taking values in the sections of  vector bundles.\\


\subsection{Random homology}
We now recall a result of \cite{lejanito} and provide a simple example. The additive semigroup of Eulerian networks is naturally mapped on the first homology group $H_1(\mathcal{G},\mathbb{Z})$ of the graph, which is defined as the quotient of the fundamental group by the subgroup of commutators. It is an Abelian group with $n=\vert E \vert -\vert X \vert +1$ generators. The homology class of the network $k$ is determined by the antisymmetric part $\widecheck{k}$ of the matrix $k$.\\The distribution of the induced random homology $\widecheck{N}^{(\alpha)}$ can be computed as a Fourier integral on the Jacobian torus of the graph $ Jac(\mathcal{G})=H^1(\mathcal{G},\mathbb{R})/ H^1(\mathcal{G},\mathbb{Z})$. \\Here, following \cite{kosu}  we denote by $H^1(\mathcal{G},\mathbb{R})$ the space of harmonic one-forms, which in our context is the space of one-forms $\omega^{x,y}=-\omega^{y,x}$ such that $\sum_{y}C_{x,y}\omega^{x,y}=0$ for all $x\in X$ and by $H^1(\mathcal{G},\mathbb{Z})$ the space of harmonic one-forms $\omega$ such that for all discrete loops (or equivalently for all non backtracking discrete loops) $\gamma$ the holonomy $\omega(\gamma)$ is an integer.\\
Precisely, if we equip $H^1(\mathcal{G},\mathbb{R})$ with the scalar product defined by the set of conductances $C$: $$\Vert \omega\Vert^2=\sum_{x,y}C_{x,y}(\omega^{x,y})^2,$$ let $d\omega$ be the associated Lebesgue measure, for all $j \in H_1(\mathcal{G},\mathbb{Z})$, and denote by $G^{(2\pi \omega)}$ the Green function attached to $P^{e^{2\pi i \omega}}$, we have:\\
 \begin{proposition} 
$$ P(\widecheck{N}^{(\alpha)}=j)=\frac{1}{\vert Jac(\mathcal{G})\vert}\int_{Jac(\mathcal{G})}\left[\frac{\det(G^{(2\pi i\omega)})}{\det(G)}\right]^{\alpha}e^{-2\pi i\langle j,\omega \rangle}d\omega.$$
 \end{proposition}
 \begin{proof}  
 Indeed, by Fourier transform
 $$P(\widecheck{N}^{(\alpha)}=j)=\frac{1}{\vert Jac(\mathcal{G}\vert}\int_{Jac(\mathcal{G})}E(e^{2\pi i\langle \widecheck{N}^{(\alpha)}-j,\omega \rangle} ) d\omega $$
$$=\frac{1}{\vert Jac(\mathcal{G}\vert}\int_{Jac(\mathcal{G})}e^{\alpha \sum_l \mu(l) (e^{2\pi i\langle \widecheck{N}(l),\omega \rangle} -1)} e^{-2\pi i\langle  j,\omega \rangle}  d\omega$$ \\
 $$=\frac{1}{\vert Jac(\mathcal{G})\vert}\int_{Jac(\mathcal{G})}\left[\frac{\det(G^{(2\pi i\omega)})}{\det(G)}\right]^{\alpha}e^{-2\pi i\langle j,\omega \rangle}d\omega.$$
\end{proof}

For $\alpha=1$, this expression can be written equivalently as $$\frac{1}{\vert \mathrm{Jac}(\mathcal{G})\vert}\int_{Jac(\mathcal{G})} E(e^{\sum_{x\neq y}(\frac{1}{2} C_{x,y} (e^{2\pi i \omega_{x,y}}-1)\varphi_x \bar{\varphi}_y)} ) e^{-2\pi i\langle j,\omega \rangle}d\omega$$
$$=\frac{1}{\vert \mathrm{Jac}(\mathcal{G})\vert}\int_{\mathrm{Jac}(\mathcal{G})} E(e^{\frac{1}{2} (\mathcal{E}-\mathcal{E}^{(2\pi i \omega)})(\varphi,\bar{\varphi}) )}e^{-2 \pi  i \langle j ,\omega \rangle} d\omega$$ where $\mathcal{E}^{(2\pi i \omega)}$ denotes the positive energy form defined by :$$\mathcal{E}^{(2\pi i \omega)}(f,g)=\frac{1}{2}\sum_{x,y}C_{x,y}(f(x)-e^{2\pi i \omega_{x,y}}f(y))(\bar{g}(x)-e^{-2\pi i \omega_{x,y}}\bar{g}(y))+\sum_x \kappa_x f^{2}(x).$$ This expression can also be written as
$$\frac{1}{\vert Jac(\mathcal{G})\vert } \frac{1}{\det(G)}\int_{Jac(\mathcal{G})} E(e^{-\frac{1}{2} \mathcal{E}^{(2\pi i \omega)}(\varphi,\bar{\varphi}) }e^{-2 \pi  i \langle j ,\omega \rangle} d\omega \dfrac{d\varphi \wedge d\bar{\varphi}}{2i}.$$\\
There is a similar expression when $\alpha$ is an integer $d$, with $d$ independent copies of the free field $\varphi$.\\

\textbf{An example:} Consider the case of the discrete circle with $N$ vertices, conductances equal to 1 and killing rate $\kappa$.\\
The homology group is $\mathbb{Z}$. $\widecheck{N}_{i,i+1}(l)$ is constant in $i$ for any loop $l$ and  $\widecheck{N}$ can therefore be viewed an an integer. Harmonic form are also constant and the Jacobian torus is $\mathbb{R} / (\mathbb{Z}/N)$. $P$ and $P^{e^{2\pi i\omega}}$ are circulant matrices and therefore, their determinants can be computed.\\ If we set $u_{\pm}=\frac{1}{2}(-1\pm\sqrt{1-\dfrac{4}{(2+\kappa)^2})},$
$$\det(I-P)=u_+^N+u_-^N+\dfrac{2(-1)^{N+1}}{(2+\kappa)^N}$$and $$\det(I-P^{e^{2\pi i\omega}})=u_+^N+u_-^N+\dfrac{(-1)^{N+1}2\cos(2\pi N\omega)}{(2+\kappa)^N}.$$ 
Hence for any integer $j$, we get 
$$P(\widecheck{N}^{(\alpha)}=j)=N \int_0^{1/N}\left[\frac{u_+^N+u_-^N+\dfrac{2(-1)^{N+1}}{(2+\kappa)^N}}{u_+^N+u_-^N+\dfrac{(-1)^{N+1}2\cos(2\pi N\omega)}{(2+\kappa)^N}}\right]^{\alpha}e^{-2\pi N j \omega }d\omega.$$
 Hence we have, with $C_N(\chi)=(-1)^{N}(2+\kappa)^N(u_+^N+u_-^N)-2,$

$$P(\widecheck{N}^{(\alpha)}=j)=\int_0^{1}\left[\frac{C_N(\kappa)}{C_N(\kappa)+2(1-\cos(2\pi\omega))}\right]^{\alpha}e^{-2\pi  j \omega }d\omega.$$\\
Note that $C_N$ is a polynomial of degree $N$ with leading order coefficient equal to $1$.\\

Letting $N$ increase to infinity with $\kappa=\frac{k}{N^2}, \:k>0$, we get that for the Brownian loop ensemble with killing rate $k$,

$$P(\widecheck{N}^{(\alpha)}=j)= \int_0^{1}\left[\frac{\cosh(\sqrt{k})-1}{\cosh(\sqrt{k})-\cos(2\pi\omega))}\right]^{\alpha}e^{-2\pi  j \omega }d\omega.$$
Factorizing the first term in the integrand, it appears that this is the distribution of the difference of two independent variables with the same negative binomial distribution of parameters $(\alpha,e^{-\sqrt{k}})$.\\

\subsection{Non Abelian holonomies}
We now consider the case of a finite, non Abelian monodromy group.\\
Given any group $G$ and $k$ of its conjugacy classes $C_1, C_2,...,C_k$, we denote $\mathcal{N}_{G}(C_1, C_2,...,C_k)$ the number of $k$-uples $(\gamma_1,\gamma_2,...\gamma_k),\:\gamma_i\in C_i$ such that $\gamma_1\gamma_2...\gamma_k=I$.
\\
Note that it is invariant by permutation of the $C_i$ and that given another class $C_0$, $\mathcal{N}_{G}(C_1, C_2,...,C_k,C_0^{-1})$ is  the number of $k$-uples whose product is in $C_0$.
 
Given a covering defined by a $M$-assignment $U$, denote by $C_U(\mathcal{L}_{\alpha})$ the set of monodromy classes defined by the discrete loops of  $\mathcal{L}_{\alpha}$ and, for any representation $\pi$ of $M$, by $\chi_{\pi}(C_U(\mathcal{L}_{\alpha}))$ the product $\prod_{l\in \mathcal{L}_{\alpha} }\chi_{\pi}(C_U(l))$.\\

The following result can be obtained as a direct generalization of lemma \ref{toto} .

 \begin{lemma} 
\begin{enumerate}

\item[i)]  With $[Z.U]_{x,y}=Z_{x,y}U_{x,y}$ we have: $$E(\prod_{x\neq y} Z_{x,y}^{N^{(\alpha)}_{x,y}}\chi_{\pi}(C_U(\mathcal{L^{(\alpha)}}))) =\left[\frac{\det(I-P^{Z.U,\pi})}{I-P}\right]^{-\alpha}.$$
\item[ii)] Moreover, for $\alpha=1$,  $$E(\prod_{x\neq y} Z_{x,y}^{N_{x,y}}\chi_{\pi}(C_U(\mathcal{L})))=E(e^{\sum_{x\neq y}\langle\frac{1}{2} C_{x,y} (Z_{x,y}U_{x,y}-I)
\varphi_x^{U,\pi} ,\:\bar{\varphi}_y^{U,\pi} \rangle}.$$
\end{enumerate}
\end{lemma} 

 For any conjugacy class $C_0$ of $M$, set $H_U^{(\alpha)}(C_0)=\frac{\mathcal{N}_M(C_U(\mathcal{L}_{\alpha}),C_0^{-1})}{\prod_{l \in \mathcal{L}_{\alpha}}  \vert C_U(l) \vert} $ . $H_U^{(\alpha)}$ is a probability on the set of conjugacy classes of $M$. It represents the proportion of product of monodromies of loops of $\mathcal{L}_{\alpha}$ which are in this conjugacy class $C_0$.
We can compute the mean value of $H_U^{(\alpha)}$ using Frobenius formula (see the appendix in \cite{LanZvo}). We get that
$$H_U^{(\alpha)}(C_0)=\sum_{\pi}\dfrac{dim(\pi)^2\chi_{\pi}(C_U(\mathcal{L}_{\alpha}))\overline{\chi_{\pi}(C_0)}}{\vert M \vert}.$$ 
Note that if $M=\mathbb{Z}/n\mathbb{Z}$, irreducible representations are given by $\pi_k(m)=e^{2\pi km}, \:k=0,1,...n-1$ and this identity reduces to:
$$1_{C_U(\mathcal{L}_{\alpha})=m_0}=\frac{1}{n}\sum_k e^{2\pi k(C_U(\mathcal{L}_{\alpha})-m_0) },$$ with $C_U(\mathcal{L}_{\alpha})=\sum_{x,y}N^{\alpha}_{x,y}U_{x,y}.$\\
Coming back to the general case, we deduce that:
$$E(H_U^{(\alpha)}(C_0))=\sum_{\pi} \dfrac{dim(\pi)^2\overline{\chi_{\pi}(C_0)}}{\vert M \vert } E \left(  e^{\sum_{l}\alpha\mu(l)(\chi_{\pi}(C_U(l)-1)}\right) .$$ Equivalently:
$$E(H_U^{(\alpha)}(C_0))=\sum_{\pi}\dfrac{dim(\pi)^2\det(G^{U,\pi})^{\alpha}\overline{\chi_{\pi}(C_0)}}{\vert M \vert \det(G)^{\alpha}}.$$
In the case of $\mathbb{Z}/n\mathbb{Z}$, if $\omega$ is such that $U_{x,y}=\omega_{x,y}\:(n)$ for all edges $(x,y)$, we obtain that:
$$P(\sum_{x,y}N^{\alpha}_{x,y}U_{x,y}=m_0)=\frac{1}{n}\sum_k e^{-2\pi k m_0}\left[\frac{\det(G^{(2\pi i\omega)})}{\det(G)}\right]^{\alpha}.$$
 Moreover, for $\alpha=1$,  $$E(H_U^{(\alpha)}(C_0))=\sum_{\pi} \dfrac{dim(\pi)^2\overline{\chi_{\pi}(C_0)}}{\vert M \vert }E\left( e^{\sum_{x\neq y}\langle\frac{1}{2} C_{x,y} (Z_{x,y}U_{x,y}-I)
\varphi_x^{U,\pi} ,\:\bar{\varphi}_y^{U,\pi} \rangle}\right) .$$

Using tensor products of representations, we can get similar formulas for all moments of $H_U$.

\section{Convergence towards Yang-Mills measure}
Given a and a loop $l_0$  $\chi_{\pi}(C_U(l_0))$ defines a  gauge-invariant functional of the  $M$-assignment $U$. Recall that given any unitary representation $\pi$ of $M$: $$E(\prod_{x\neq y} \chi_{\pi}(C_U(\mathcal{L^{(\alpha)}}))) =\left[\frac{\det(I-P^{U,\pi})}{I-P}\right]^{-\alpha}=e^{\sum_{l}\alpha\mu(l)(\chi_{\pi}(C_U(l))-1)}.$$
Denote this quantity by $\Lambda_{\pi}^{\alpha}(U)$. We see that $\Lambda_{\pi}^{\alpha}$ defines a measure on $M-$coverings, as it is a measure on the set of $M$-assignments invariant under the action of the gauge group.

Assume now that all loops with non trivial homotopy contain $d$ edges or more. In a square or cubic lattice for example, we have $d=4$. For a general graph, let us still call these loops of minimal length plaquettes and denote by $\mathcal{P}$ the set of plaquettes.

For any $c>0$, Yang Mills measures can be defined on $M$-coverings by the weights $\Lambda_{\pi,c}(U)=e^{-c\sum_{l\in \mathcal{P}}\mu(l)\chi_{\pi}(C_U(l)-1)}$ (see \cite{Seil}).

Let now $\varepsilon$ be a parameter converging to $0$. If we add $\frac{\lambda(1-\varepsilon)}{\varepsilon}$ to $\kappa$ so that $\lambda$ is divided by $\varepsilon$, and take $\alpha=c\varepsilon^{-d}$, we see that:
\begin{proposition}  As $\varepsilon\longrightarrow 0$, the weights $\Lambda_{\pi}^{\alpha}(U)$ converge to $\Lambda_{\pi,c}(U)$.
\end{proposition} 
Indeed if $\varepsilon$ is small enough, the contribution of the loops with length strictly larger than $d$ can be bounded by a geometric series whose sum is of order $C\varepsilon$, $C$ being some constant. 

This measure on coverings can be extended to the case of compact monodromy groups. The space of coverings can be replaced by the set of connections, i.e the quotient of the group of $M$-assignments by the action of the gauge group (which acts by conjugacy). 

If $M=U(1)$ , the set of connections  can be identified with the Jacobian torus $Jac(\mathcal{G})$. We can choose $\pi$ to be the identical representation $\iota$ and then, $$\Lambda_{\iota}^{\alpha}(\omega)= \left[\frac{\det(I-P^{e^{2\pi i\omega}})}{\det(I-P)}\right]^{-\alpha}=\left[\frac{\det(G^{(2\pi i\omega)})}{\det(G)}\right]^{\alpha}.$$
In the case of our elementary example on the circle, we find that :
$$\Lambda_{\iota}^{\alpha}(\omega)=\left[\frac{C_N(\kappa)}{C_N(\kappa)+2(1-\cos(2\pi\omega))}\right]^{\alpha}$$
and that
$$\Lambda_{\iota,c}(\omega)=\exp \left(-2c\dfrac{(1-\cos(2\pi\omega))}{(2+\kappa)^N}\right).$$
Note finally that we can study in parallel the random spanning tree on the covering, its projection to $\mathcal{G}$ and the associated fermionic fields (see \cite{stfl}). Any probability on coverings produces a coupling at the level of loops and trees and therefore at the level of Gaussian (i.e. bosonic) and fermionic free fields. We plan to study this in more detail in a forthcoming work.\\

\textbf{Acknowledgment:} Hearty thanks are due to the referee for his careful reading and useful suggestions.

\bigskip

\noindent

  D\'epartement de Math\'ematique. Universit\'e Paris-Sud.  Orsay, France.

\bigskip
   yves.lejan@math.u-psud.fr

\end{document}